%% file: 3space_increase_loc_con190705.tex
\documentclass[11pt]{article}

\input{minimalJ}

\usepackage{verbatim}
\usepackage{enumerate}
\usepackage[all]{xy}

\usepackage{subfig}

\title{Embedding simply connected \\ 2-complexes in 3-space\\ \Large V. A refined Kuratowski-type 
characterisation}

\author{Johannes Carmesin
\medskip 
\\
  {University of Birmingham}
}

\mcm{\Fbb}{0}{\mathbb{F}}

\usepackage{xr}
\begin{document}

\maketitle

\begin{abstract}
This paper is the last paper in a series of five papers. Building on earlier papers in this series, 
we prove an analogue of Kuratowski's characterisation of graph planarity for three 
dimensions. 

More precisely, a simply connected 2-dimensional simplicial complex embeds in 3-space 
if and only if it has no obstruction from an explicit list. This list of obstructions is finite 
except for one infinite family. 
\end{abstract}

\section{Introduction}

We assume that the reader is familiar with \cite{3space1}. In that paper we prove that a locally 
3-connected simply connected 2-dimensional simplicial complex has a 
topological embedding into 3-space if and only if it has no space minor from a finite 
explicit list $\Zcal$ of obstructions. The purpose of this paper is to extend that theorem beyond 
locally 3-connected (2-dimensional) simplicial complexes to simply connected
simplicial complexes in general. 

\vspace{0.3 cm}

The first question one might ask in this direction is whether the assumption of local 
3-connectedness could simply be dropped from the result of \cite{3space1}. Unfortunately this is not 
true. One new obstruction can be constructed from the M\"obius-strip as follows. 

Consider the central cycle of the M\"obius-strip, see \autoref{fig:moe}. Now attach a disc at that 
central 
cycle. In a few lines we explain why this topological space $X$ cannot be embedded in 3-space. Any 
triangulation of $X$ gives an obstruction to embeddability. It can be shown that such 
triangulations have no space minor in the finite list $\Zcal$.

   \begin{figure} [htpb]   
\begin{center}
   	  \includegraphics[height=3cm]{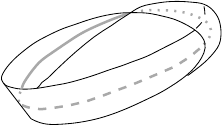}
   	  \caption{The  M\"obius-strip. The central cycle is depicted in grey.}\label{fig:moe} 
\end{center}
   \end{figure}

Why can $X$ not be embedded in 3-space? To answer this, consider a small torus around the 
central cycle. The disc and the M\"obius-strip each intersect that torus 
in a circle. These circles however have a different homotopy class in the torus. Since any two 
circles in the torus of a different homotopy class intersect\footnote{A simple way to see this is 
to note that the torus with a circle removed is an annulus. }, the space $X$ cannot 
be embedded in 3-space without intersections of the disc and the M\"obius-strip. Obstructions of 
this type we call \emph{torus crossing obstructions}. A precise definition is given in 
\autoref{s0}. 

A refined question might now be whether the result of  \cite{3space1} extends to simply connected
simplicial complex if we add the list $\Tcal$ of torus crossing obstructions to the list $\Zcal$ of 
obstructions. The answer to this question is `almost yes'. Indeed, we just need to add to the space 
minor operation the operations of stretching defined in \autoref{s3}. 
These operations are illustrated in \autoref{fig:s_pair}, \autoref{fig:stretch_branch} and 
\autoref{fig:stretch_edge}.

It is not hard to show that 
stretching preserves embeddability. The 
main result of this paper is the following. 

\begin{thm}\label{Kura_gen}\label{Kura_simply_con}
  Let $C$ be a simply connected simplicial complex. The following are equivalent.
 \begin{itemize}
  \item $C$ has a topological embedding in 3-space;
  \item $C$ has no stretching that has a space minor in $\Zcal\cup \Tcal$. 
 \end{itemize}
\end{thm}

We deduce \autoref{Kura_gen} from the results of \cite{3space1} in two steps as follows.
The notion of `local almost 3-connectedness and stretched out' is slightly more general and more 
technical than `local 3-connectedness', see \autoref{s0} for a definition. First we extend the 
results of 
\cite{3space1} to locally almost 3-connected and stretched out simply connected
simplicial complexes, see \autoref{Kura_simply_con2} below. 
We conclude the proof by 
showing that any simplicial complex can be stretched to a locally almost 3-connected and 
stretched out one. More 
precisely:

\begin{thm}\label{main_streching}
 For any simplicial complex $C$, there is a simplicial complex $C'$ obtained from $C$ by stretching 
so that $C'$ is 
locally almost 3-connected and stretched out or $C'$ has a non-planar link.

Moreover $C$ has a planar rotation system if and only if $C'$ has a planar rotation system.
\end{thm}

The overall structure of the argument is similar to that for problems in structural graph theory 
with `3-connected kernel' (in such arguments one first proves the 3-connected case, then in a 
second step deduces the 2-connected case and then finally deduces the general case).
In \autoref{s0} we prove 
the extension of the results of \cite{3space1} to locally almost 3-connected and stretched out 
simplicial complexes, \autoref{Kura_almost}. 
In \autoref{s2} we develop the tools to extend this to the locally almost 2-connected case. 
In \autoref{s1} and \autoref{s3}  we extend \autoref{Kura_almost} to general simplicial complexes, 
which proves \autoref{main_streching}. Then we prove 
\autoref{Kura_simply_con}. Finally in \autoref{algo_sec} we describe algorithmic consequences. 

For graph theoretic definitions we refer the 
reader to \cite{DiestelBookCurrent}.

\section{A Kuratowski theorem for locally almost 3-connected simply connected simplicial 
complexes}\label{s0}

In this section we prove \autoref{Kura_simply_con2}, which is used in the proof of the main theorem.
First we define the list $\Tcal$ of torus crossing obstructions. 

Given a simplicial complex $C$, a \emph{mega face} $F=(f_i|i\in \Zbb_n)$ is a cyclic 
orientation of faces $f_i$ of $C$ together with for every $i\in \Zbb_n$ an edge 
$e_i$ of $C$ that is only incident with $f_i$ and $f_{i+1}$ such that the $e_i$ and $f_i$ are 
locally distinct, that is, $e_i\neq e_{i+1}$ and $f_i\neq f_{i+1}$ for all $i\in \Zbb_n$. We remark 
that since in a simplicial 
complex any two faces can share at most one edge, the edges $e_i$ are implicitly given by the faces 
$f_i$. \begin{comment}
        
In our notation we very often suppress the edges $e_i$. 
       \end{comment}
A \emph{boundary component} of a mega 
face $F$ is a connected component of the 1-skeleton of $C$ restricted to the faces $f_i$ after we 
delete the edges $e_i$. 
Given a cycle $o$ that is a boundary component of a mega face $F$, 
we say that $F$ is 
\emph{locally monotone} at $o$ if for every edge $e$ of $o$ and each face $f_i$ containing $e$, the 
next face of $F$ after $f_i$ that contains an edge of $o$ contains the unique edge of $o$ that has 
an endvertex in common with $e$ and $e_{i+1}$.
Under these assumptions for each edge $e$ of $o$ the number of indices $i$ such that $e$ is 
incident with $f_i$ is the same. This number is called the  \emph{winding 
number} of 
$F$ at $o$.

A \emph{torus crossing obstruction} is a simplicial complex $C$ with a cycle $o$ (called the 
\emph{base cycle}) whose faces can be partitioned into two mega faces that both have $o$ has a 
boundary component and are locally monotone at $o$ but with different winding numbers. We denote 
the set of torus crossing 
obstructions by $\Tcal$.

\begin{rem}The set of torus crossing obstructions is infinite. Indeed, it contains at least one 
member for every pair of distinct winding numbers. So it is not possible to reduce it to a finite 
set. However one can further reduce torus crossing obstruction as follows. First, by working with 
the class of 3-bounded 2-complexes as defined in \cite{3space1} instead of simplicial complexes, one 
may assume that the cycle $o$ is a loop. Secondly, one may introduce the further operation of gluing 
two faces along an edge if that edge is only incident with these two faces. This way one can glue 
the two mega faces into single faces. Thirdly, one can 
enlarge the holes of the mega faces to make them into one big hole (after contracting edges 
one may assume that this single hole is bounded by a loop). After all these steps we 
only have one torus crossing obstruction left for any pair of distinct winding numbers. This 
obstruction consists of three vertex-disjoint loops and two faces, each incident with two loops. The 
loop contained by both faces is the base cycle $o$. Here the faces may have winding number greater 
than one. The faces have winding number precisely one at the other loops. 
\end{rem}

A \emph{parallel graph} consists of two vertices, called the \emph{branch vertices}, and a set of 
disjoint paths between them. Put another way, start with a graph with only two vertices and all 
edges going between these two vertices, now subdivide these edges arbitrarily, 
see \autoref{fig:para5}. 

   \begin{figure} [htpb]   
\begin{center}
   	  \includegraphics[height=2cm]{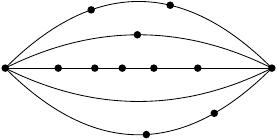}
   	  \caption{A parallel graph with five paths.}\label{fig:para5} 
\end{center}
   \end{figure}
   For example, parallel graphs where the branch vertices have degree two are cycles. 

Given a simplicial complex $C$ and a cycle $o$ of $C$, we say that $o$ is a 
\emph{para-cycle} 
if all link graphs at the vertices of $o$ are parallel graphs.
\begin{lem}\label{has_obstruction}
  Let $C$ be a simplicial complex. Assume that $C$ has a para-cycle $o$ such that for some 
edge 
$e$ of $o$ the link graph $L$ of the contraction $C/(o-e)$  at the vertex $o-e$ is not loop planar. 
Then a torus crossing obstruction can be obtained from $C$ by deleting faces. 
\end{lem}

\begin{proof}
Our aim is to define a torus crossing obstruction with base cycle $o$. For that we define a set of 
possible mega faces as follows.

The complex $C/(o-e)$ has only one loop and that is $e$. We denote the two vertices of $L$ 
corresponding to $e$ by $\ell_1$ and $\ell_2$. 
 Since $o$ is a para-cycle, 
the 
link graph $L$ is (isomorphic to) a parallel graph with branching vertices $\ell_1$ and 
$\ell_2$. We shall define mega faces such that every edge of the parallel graph incident 
with 
$\ell_1$ is a face of precisely one of these mega faces. We define these mega faces recursively. So 
let $f$ be an edge of the parallel graph incident with $\ell_1$ that is not already assigned 
to a mega face. Let $P$ be the path of the parallel graph between $\ell_1$ and $\ell_2$ that 
contains $f$. The 
edges on that path after $f$ are its consecutives in its mega face. The last edge of that path is 
incident with $\ell_2$ and hence it also corresponds to an edge incident with $\ell_1$. If that 
face is equal to $f$ we stop. Otherwise we continue with that face as we did with $f$, see 
\autoref{megafig}.

   \begin{figure} [htpb]   
\begin{center}
   	  \includegraphics[height=3cm]{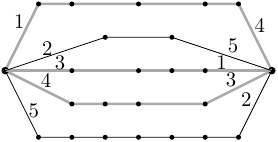}
   	  \caption{The construction of a mega face in a subdivision of $B_5$. The bijection between 
the edges incident 
with $\ell_1$ and $\ell_2$ is indicated by numbers. In grey we marked a set of the edges whose 
faces form a mega face. }\label{megafig} 
\end{center}
   \end{figure}

Eventually, 
we will come back to the face $f$. This completes the definition of the mega face containing $f$. 
This defines a mega face as all interior vertices of these paths have degree two. 
It is clear from this definition that the mega faces partition the edges of the link graph. 
Since $o$ is a para-cycle, these mega-faces are also mega-faces of $C$ and the cycle $o$ is a 
boundary component of each of them. It is straightforward to check that these mega-faces are 
monotone at $o$. 

It suffices to show that two of these mega faces have distinct winding number at $o$.
Suppose not for a contradiction. Then all mega faces have the same winding number.

We enumerate the mega faces and let $K$ be their total number. 
The winding number of a mega face is equal to the number of its traversals of the edge $e$, that 
is, its number of faces that -- when considered as edges of the link graph -- are incident with 
$\ell_1$.
So by our assumption, there is a constant $W$ such that all our mega faces contain 
precisely $W$ faces incident with $e$. We enumerate these faces in a subordering of the mega 
face. More precisely, by $f[k,w]$ we denote the $w$-th face incident with $e$ on the 
$k$-th mega face, where $k$ and $w$ are in the cyclic groups $\Zbb_K$ and $\Zbb_W$, respectively.

We will derive a contradiction by constructing a rotation system of the link graph $L$ that is loop 
planar. 
We embed it in the plane such that the rotation system at $\ell_1$ is $f[1,1]$, 
$f[2,1]$,  \ldots, $f[K,1]$, $f[1,2]$,  $f[2,2]$, \ldots, $f[K,2]$, $f[1,3]$, \ldots ,\ldots, 
$f[K,W]$, $f[1,1]$. 

Then the rotation system at $\ell_2$ is obtained from the that of $\ell_1$ by replacing each 
face  $f[k,w]$ by $f[k,w+1]$ and then reversing. Since this shift operation keeps this particular 
cyclic ordering invariant, the rotation systems at $\ell_1$ and $\ell_2$ are reverse. So this 
defines a loop planar embedding of the link graph. Hence $L$ has a loop planar rotation system. 
This is the 
desired contradiction to our assumption. 
Hence two mega faces must have a different winding number. So $C$ contains a torus crossing 
obstruction. 
\end{proof}

\begin{rem}Next we define `stretched out'. 
 This is a technical condition, which is used only twice in the 
argument, 
namely in the proof of \autoref{Bm_path} and \autoref{41_2} below. We remark that the notion of 
stretched out as defined here is only intended to be useful for locally 3-connected 2-complexes. 
There it very roughly says that every edge of degree two has an endvertex whose link graph is every 
simple. 
\end{rem}

A simplicial complex is \emph{stretched out} if 
 every edge incident with only two faces has an endvertex $x$ such that the link graph at $x$ 
is not a subdivision of a 3-connected graph and not a parallel graph whose branching vertices 
have degree at least three.

Next we define `para-paths', which are similar to para-cycles and analyse them. 
A path in a simplicial complex $C$ is a \emph{para-path} if 
\begin{enumerate}
 \item the link graphs at all interior vertices of $P$ are parallel graphs, where the branching 
vertices have degree at least three;
\item the link graphs at the two endvertices of $P$ are subdivisions of 3-connected graphs.
\end{enumerate}

\begin{lem}\label{Bm_path}
Let $C$ be a stretched out simplicial complex\footnote{In this paper we follow the convention that 
every edge of a simplicial complex is incident with some face.} with a para-path $P$. Then the 
complex $C'$ obtained from $C$ by contracting all edges of the path $P$ has at most one loop.
\end{lem}

\begin{proof}
 Let $v$ and $w$ be the endvertices of the path $P$. Since $C$ is a simplicial complex, there is at 
most one edge between $v$ and $w$. We will show that no other edge of $C$ becomes a loop in $C'$. 

So let $e$ be an edge of $C$ that has an endvertex $u$ on the path $P$ different from 
$v$ and $w$. Thus the vertex $u$ is an interior vertex of $P$, so $L(u)$ is a parallel graph whose 
branching vertices have degree at least three.
As $C$ is stretched out, the other endvertex $x$ of $e$ has a link graph different from all link 
graphs at vertices on the para-path $P$. Thus $x$ does not lie on the para-path $P$. Thus 
the edge $e$ is not a loop in the simplicial complex $C'$. 
\end{proof}

   \begin{figure} [htpb]   
\begin{center}
   	  \includegraphics[height=2cm]{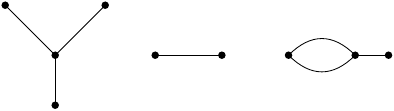}
   	  \caption{A 3-star, an edge and a 2-cycle with an 
attached leaf. Free-graphs are subdivisions of these graphs.}\label{fig:starstarround} 
\end{center}
   \end{figure}

A \emph{free-graph} is a subdivision of a 3-star, a path or a cycle with an attached path, 
see \autoref{fig:starstarround}.
These graphs are `free' in the sense that any rotation system on them defines an embedding in the 
plane. 
A graph is \emph{almost 3-connected} if it is a subdivision of a 3-connected graph, a parallel 
graph or a free-graph. 
A simplicial complex is \emph{locally almost 3-connected} if all its link 
graphs are almost 3-connected. 

\begin{thm}\label{Kura_simply_con2}\label{Kura_almost}
 Let $C$ be a simplicial complex that is locally almost 3-connected and stretched out. The 
following are equivalent.
 \begin{itemize}
  \item $C$ has a planar rotation system;
  \item $C$ has no space minor in $\Zcal\cup \Tcal$. 
 \end{itemize}
\end{thm}

As a preparation for the proof of \autoref{Kura_almost}, we prove the following analogue of 
{\cite[\autoref*{rot_system_exists}]{3space1}}. Recall that an edge $e$ is a \emph{chord} of a 
cycle 
$o$ in a 
simplicial complex if $e$ is not in $o$ but joins two vertices of $o$.

\begin{lem}\label{41_2}
  Let $C$ be a simplicial complex that is locally almost 3-connected and stretched out. Then $C$ 
has a planar 
rotation system unless
\begin{enumerate}
 \item $C$ is not locally planar;
 \item there is a para-path $P$ such that $C/P$ is not locally planar at the vertex $P$;
 \item the contraction $C/(o-e)$ is not locally planar, where $o$ is a chordless cycle and $e$ is 
an edge of 
$o$ and $o$ contains an edge aside from $e$.
\end{enumerate}
\end{lem}

\begin{proof}
We obtain $H$ from the 1-skeleton of $C$ by deleting all edges of $C$ that are incident with 
precisely two faces or such that the link graph at one endvertex is a free-graph. 
Let $H'$ be a connected component of $H$. We say that a rotation system of $C$ is \emph{planar at 
$H'$} if it is planar at all vertices of $H'$. 
In order to show 
that $C$ has a planar rotation system, it suffices to 
construct for each connected component $H'$ of $H$ a rotation system of $C$ that is 
planar at $H'$.
Indeed, since the rotators at vertices of degree two are unique, 
we can combine these rotation systems for the different components of $H$ to a planar rotation 
system of $C$.  And if the link graph at one endvertex $v$ of an edge $e$ is a free-graph, we can 
just change the rotator at $e$ in $L(v)$ so as to be the reverse of the rotator at the link graph 
at $e$ in the other link graph containing $e$.

First assume that $H'$ just consists of a single vertex. 
Either $C$ has a rotation system that is planar at $H'$ or the link graph of $C$ at the single 
vertex of $H'$ 
is not loop planar. That is, we have the first outcome of the 
lemma. 

Note that vertices whose link graphs are free-graphs are included in this case, as they do not 
have 
any outgoing edges in $H$. 

Next assume that all link graphs at vertices of $H'$ are parallel-graphs. Since we may assume that 
$H'$ contains at least two vertices, each branching vertex of such a parallel graph has degree 
at least three; and each vertex of $H'$ is incident with precisely two edges (which are the 
branching vertices in its link graph). So the connected graph $H'$ is a cycle $o$. In fact, it is a 
para-cycle. 

Our aim is to show that there is a rotation system  planar at $H'$ or we get outcome 3 from the 
lemma. For that we contract the edges of $o$ one by one until a single edge $e$ remains. 
After each contraction one gets a para-cycle with one fewer vertex. 
Similarly as {\cite[\autoref*{contr_pres_planar}]{3space1}} one proves that there is a 
rotation 
system planar at $H'$ before the contraction of a single edge 
if and only if there is such a  planar rotation 
system after the contraction. Thus $C$ has a rotation system planar at $H'$ or the 
2-complex $C/(o-e)$ is not loop planar at $o-e$. That is, we have the third outcome 
of the lemma, as the cycle $o$ of the stretched out simplicial complex $C$ is chordless. 

\vspace{.3cm}

Thus it suffices to consider the case that $H'$ contains a vertex whose link 
graph is a subdivision of a 3-connected 
graph.
Let $X$ be the set of those vertices of $H'$ where the link graph is a subdivision of a 3-connected 
graph. 

\begin{rem}
 If all vertices of $H'$ are in $X$, the proof of {\cite[\autoref*{rot_system_exists}]{3space1}} 
extends almost verbatim to this case. We reduce the general case to 
{\cite[\autoref*{rot_system_exists}]{3space1}} as follows. 
\end{rem}

Vertices of $H'$ not in $X$ have parallel graphs at their links. These 
vertices have degree one or two in $H'$. (Indeed, as they are in $H'$ they must be 
parallel 
graphs whose (two) branching vertices have degree at least three.) 

Thus the graph $H'$ consists of the set $X$ together with some paths between these vertices, 
or cycles and paths attached at single vertices of $X$. These paths starting at a 
single vertex of $X$ must have a deleted edge $d$ of degree three incident with their last vertex. 
So the link graph at the other endvertex of $d$ is a free graph. 
We obtain $H''$ from $H'$ by deleting all the paths attached at a single vertex of $X$; here we 
stress that we do not delete their starting vertex in $X$ and we do not delete attached cycles. 
Note that there is a rotation system that is planar for $H'$ if and only if there is a rotation 
system that is planar for $H''$. 

Let $C'$ be the simplicial complex obtained from $C$ 
contracting all but one edge from every path of $H''$ between two vertices of $X$ or every cycle of 
$H''$ containing precisely one vertex of $X$. 
Then $C'$ is a 3-bounded 2-complex such that the link graph at every vertex of $X$ is a subdivision 
of a 3-connected graph. 
As $C$ is stretched out, no edge of degree two is a loop in $C'$. The 2-complex $C'$ has one loop 
for every cycle of $H''$ containing a single vertex of $H$, and no further loops by construction.

As contractions of non-loops and their inverse operations preserve the 
existence of planar rotation systems in locally 2-connected 2-complexes by 
{\cite[\autoref*{contr_pres_planar}]{3space1}} and local 
2-connectivity is preserved by 
contraction by {\cite[\autoref*{sum_3con}]{3space1}}, there is a planar rotation system 
for $H''$ in $C$ if and only 
if there is a planar rotation system for $X$ in $C'$. The same proof of as that of 
{\cite[\autoref*{rot_system_exists}]{3space1}} gives that there is a planar rotation system for $X$ 
in $C'$ unless one of the following occurs.
\begin{enumerate}
\item $C'$ is not locally planar;
 \item there is a non-loop $e$ of $C'$ such that $C/e$ is not locally planar at the vertex $e$;
 \item the contraction $C'/(o-e)$ is not locally planar, where $o$ is a chordless cycle and $e$ is 
an edge of 
$o$ and $o$ contains an edge aside from $e$.
\end{enumerate}
In the first case, let $v'$ be the vertex of $C'$ whose link graph is not loop-planar. 
Let $v$ be the unique vertex of $X$ contracted onto $v'$. If the link graph at $v$ of $C$ is not 
planar, we have outcome 1 of \autoref{41_2}. 
Hence we may assume that this is not the case, and so the link graph $L(v)$ is planar, and so also 
$L(v')$ is planar -- but not loop-planar. As the link graph $L(v')$ is a subdivision of a 
3-connected graph, there is a single loop $e$ of $C'$ that witnesses that the link graph $L(v')$ is 
not loop planar. Let $o$ be the unique cycle of $H''$ containing $e$. Then the 2-complex 
$C/(o-e)$ is not loop planar at the contraction vertex. Hence we have outcome 3 of \autoref{41_2}, 
as the cycle $o$ of the stretched out simplicial complex $C$ is chordless. 

In the second case, the edge $e$ is a path of $C$ all whose interior vertices have parallel graphs 
at their links. Hence we get a para-path as in outcome 2 of 
\autoref{41_2}.

In the third case, each edge of the cycle $o$ is a path of $C$, and all these paths together form 
a cycle of $C$. This cycle has no chord as $o$ has no chord. We pick an arbitrary 
edge on the path for $e$, and we get outcome 3 of \autoref{41_2}.
\end{proof}

\begin{proof}[Proof of \autoref{Kura_simply_con2}.]
By \autoref{has_obstruction} we may assume that $C$ has no para-cycle $o$ such that for some 
edge $e$ of $o$ the contraction $C/(o-e)$ is not loop planar at the vertex $o-e$.

Next we treat the case that $C$ has a para-path $P$ such that the link graph $L(P)$ of $C/P$ 
at 
$P$ is not loop planar. The link graph $L(P)$ is the vertex-sum of the link graphs at the vertices 
of $P$. Thus it is a subdivision of a 
3-connected graph by {\cite[\autoref*{sum_3con}]{3space1}}. 
By \autoref{Bm_path}, $C/P$ has at most one loop, which is incident with $L(P)$. 
By  
{\cite[\autoref*{reduce_to_looped_cone1}]{3space1}} or 
{\cite[\autoref*{reduce_to_looped_cone2}]{3space1}} $C'$ has a space minor that is a generalised 
cone or a looped generalised cone that is not loop 
planar at its top, respectively.  In the 
first case we deduce by 
{\cite[\autoref*{gen_cone}]{3space1}} that $C'$ has a space minor in $\Zcal_1$. In the second case 
we deduce similarly as in the last paragraph of the proof of  
{\cite[\autoref*{rot_minor}]{3space1}} 
 that 
$C'$ has a space minor in $\Zcal_2$.

Having treated the above cases the rest of the proof of \autoref{Kura_simply_con2} is analogue 
to the proof of  {\cite[\autoref*{rot_minor}]{3space1}} 
except that we 
refer to \autoref{41_2} instead of {\cite[\autoref*{rot_system_exists}]{3space1}}.
\end{proof}

\section{Streching local 2-separators}\label{s2}

In this section we define stretching at local 2-separators and prove basic properties of this 
operation. 
This operation is necessary for \autoref{main_streching}.  

A \emph{2-separator} in a 2-connected graph\footnote{In this paper we will only consider 
2-separators of link graphs of simplicial complexes; such link graphs do not have parallel 
edges or loops. For multigraphs, it seems suitable to also consider $(a,b)$ a 2-separator 
if there are two parallel edges between them and $L-a-b$ is 
not empty or $a$ and $b$ have three parallel edges in between. } $L$ is a 
pair of 
vertices $(a,b)$ such that 
$L-a-b$ has at least two connected components.

Given a simplicial complex $C$ with a vertex $v$ such that its link graph $L(v)$ is 2-connected  
and has a 2-separator $(a,b)$, the simplicial 
complex $C_2$ obtained from $C$ by 
\emph{stretching $\{a,b\}$ 
at $v$} is defined as follows, see \autoref{fig:s_pair}.

   \begin{figure} [htpb]   
\begin{center}
   	  \includegraphics[height=4.5cm]{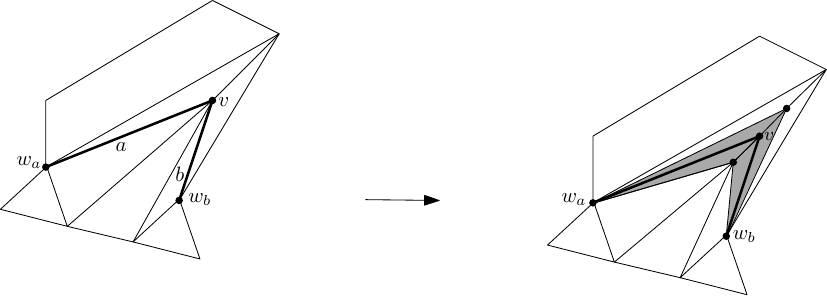}
   	  \caption{If we stretch the highlighted pair of edges in the simplicial complex on the 
left, 
we obtain the one on the right. The newly 
added faces are 
depicted in grey.}\label{fig:s_pair} 
\end{center}
   \end{figure}

   \begin{figure} [htpb]   
\begin{center}
   	  \includegraphics[height=1cm]{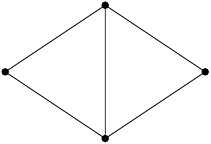}
   	  \caption{The simplicial complex $\Delta_2$.}\label{fig:Delta2} 
\end{center}
   \end{figure}

We denote by $\Delta_2$ the simplicial complex obtained from two disjoint faces of size three by 
gluing them together at an edge, see \autoref{fig:Delta2}. 
Let $\Delta^+_n$ be the simplicial complex obtained by gluing $n$ copies of $\Delta_2$ 
together at a path of length two whose endvertices have degree two in $\Delta_2$ (this is uniquely 
defined up to isomorphism), see \autoref{fig:Delta2N}. 
   \begin{figure} [htpb]   
\begin{center}
   	  \includegraphics[height=2cm]{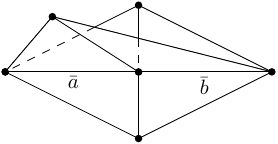}
   	  \caption{The simplicial complex $\Delta^+_3$ with the gluing edges  
labelled $\bar a$ and $\bar b$.}\label{fig:Delta2N} 
\end{center}
   \end{figure}

Informally, we obtain $C_2$ from $C$ by replacing the edges $a$ and $b$ by $\Delta^+_n$, where $n$ 
is the 
number of components of $L(v)-a-b$. 
More precisely, the simplicial complex $C_2$ is defined as follows. 
Let $n$ be the number of components of $L(v)-a-b$. 
We denote the gluing edges of $\Delta^+_n$ by 
$\bar a$ and $\bar b$. We label the vertices of $\Delta^+_n$ incident with neither $\bar a$ nor 
$\bar b$ by the components of $L(v)-a-b$.

In our notation we suppress a bijection between vertices of $C$ and $\Delta^+_n$ as follows.
We label the common vertex of the edges of $\bar a$ and $\bar b$ by $v$.
We denote the endvertex of the edge $a$ in $C$ different from $v$ by $w_a$; and we label the 
endvertex of the edge $\bar a$ different from $v$ by $w_a$. Similarly, we denote the endvertex of 
the edge $b$ in $C$ different from $v$ by $w_b$; and we label the 
endvertex of the edge $\bar b$ different from $v$ by $w_b$.
\begin{itemize}
 \item The vertex set of $C_2$ is union of the vertex set of $C$ together with 
the vertex set 
of $\Delta^+_n$, in formulas: $V(C_2)=V(C) \cup V(\Delta^+_n)$. We stress that the sets $V(C)$ and 
$V(\Delta^+_n)$ share the vertices $v$, $w_a$ and $w_b$ and hence these vertices appear in $V(C_2)$ 
only once as $V(C_2)$ is just a set and not a multiset;
\item the edge set of $C_2$ is (in bijection with) the edge set of $C$ with the edges $a$ and 
$b$ replaced by 
the set of edges of $\Delta^+_n$, in formulas: $E(C_2)=\left (E(C)-a-b\right ) \cup E(\Delta^+_n)$.
The incidences between vertices and edges are as in $C$ or $\Delta^+_n$, except for those edges of 
$C$
that have the vertex $v$ as an endvertex.
This defines all incidences of edges except those of $C$ that have the endvertex $v$.
Given an edge $x$ of $C$ incident with $v$, and denote its other endvertex by $x'$.
Then its corresponding edge of $C_2$ has the endvertices $x'$ and the vertex of $\Delta^+_n$ that 
is 
the component of $L(v)-a-b$ containing $x$. This completes the definition of the edges of $C_2$. 
We stress that the vertex $w_a$ of $C_2$ is incident with those edges of $C-a-b$ with endvertex 
$w_a$ and 
those edges of $\Delta^+_n$ with endvertex $w_a$;
\item the faces of $C_2$ are the faces of $C$ together with the faces of $\Delta^+_n$; in formulas: 
$F(C_2)=F(C)\cup F(\Delta^+_n)$. We stress that the sets $F(C)$ and $F(\Delta^+_n)$ are disjoint. 
The incidences between edges and faces are as in $C$ or $\Delta^+_n$, where defined. This defines 
all 
incidences of faces except for those faces $f$ of $C$ incident with the edges $a$ or $b$, which are 
defined as follows. 
There are three cases:

-- if $f$ is a face of $C$ incident with both the edges $a$ and $b$, then in $C_2$ these incidences 
are replaced by incidences with the edges $\bar a$ and $\bar b$; 

-- if $f$ is a face of $C$ incident with the edge $a$ but not $b$, then in $C_2$ the incidence of 
$f$ with $a$ is replaced with an incidence with the edge $w_ax$ of $\Delta^+_n$; where 
$x$ 
is the component of $L(v)-a-b$ such that in $L(v)$ the edge $f$ joins $a$ with a 
vertex of $x$;

-- similarly, if $f$ is a face of $C$ incident with the edge $b$ but not $a$, then in $C_2$ the 
incidence of 
$f$ with $b$ is replaced with an incidence with the edge $w_bx$ of $\Delta^+_n$; where 
$x$ 
is the component of $L(v)-a-b$ such that in $L(v)$ the edge $f$ joins $b$ with a 
vertex of $x$.
\end{itemize}

This completes the definition of stretching a 2-separator at a 
vertex. 

We refer to the vertices of $C_2$ that are not in $V(C)-v$ as the \emph{new vertices}, other 
vertices of $C_2$ are called \emph{old}.

The link graph at $w_a$ of $C$ is obtained 
from the link graph at $w_a$ in $C_2$ by contracting all edges incident with the vertex $\bar a$. 
Note that $w_a$ cannot be incident with $b$ as $C$ is a simplicial complex. 

\begin{eg}
 In \autoref{fig:link2} we explain how the link graphs of 
\autoref{fig:s_pair} change.
   \begin{figure} [htpb]   
\begin{center}
   	  \includegraphics[height=3cm]{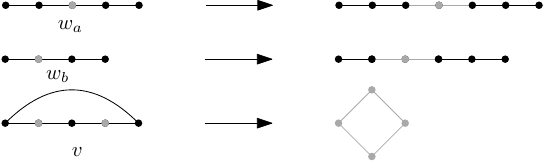}
   	  \caption{On the left we see the link graphs at the vertices $w_a$, 
$w_b$ and $v$ of the simplicial complex in \autoref{fig:s_pair}. On the right we see the link 
graphs after the stretching at $(a,b)$. The vertices $a$ and $b$ and the new vertices and edges in 
the link graphs are depicted in grey. }\label{fig:link2} 
\end{center}
   \end{figure}
   
\end{eg}

The \emph{(abbreviated) degree-sequence} of a graph is the sequences of degrees of its vertices, 
ordered by size, where we leave out the degrees which are at most two. We compare degree-sequences 
in the lexicographical order.

\begin{lem}\label{2-stretch-change}
 Let $C_2$ be a simplicial complex obtained from $C$ by stretching the 2-separator $(a,b)$ at $v$.
 Then at each vertex of $C$ aside from $v$, the degree-sequence at its link in $C_2$ is at most 
the degree-sequence at its link in $C$.
At all new vertices of $C_2$ the degree-sequence at the link is strictly smaller than the 
degree-sequence of the link graph at $v$ in $C$ -- unless the link graph at $v$ in $C$ is a 
parallel 
graph or $L(v)-a-b$ has two components and one is a path. 
\end{lem}

\begin{proof}
Coadding a star at a vertex cannot increase the abbreviated degree-sequence, hence the lemma is 
true at old vertices of $C_2$. 
So it remains to prove the lemma for the new vertices of $C_2$ as it is obvious at the 
others. As the link graph $L(v)$ at $v$ in $C$ is not a parallel graph, the degree-sequence at the 
link at $v$ in $C_2$ is strictly smaller than that in $C$.
Now let $X$ be a component of $L(v)-a-b$. If $L(v)-a-b$ has at least three components, then  
 the degree-sequence at the 
link at $X$ in $C_2$ is strictly smaller the degree-sequence of the link graph at $v$ in 
$C$. This is also true if $L(v)-a-b$ has only two components and the other component has a vertex 
of degree greater than two; that is, is not a path. This completes the proof of the lemma.  
\end{proof}

The \emph{degree-parameter} of a 2-complex is the sequence of degree-sequences of all its link 
graphs, ordered by size.  We compare degree-parameters 
in the lexicographical order.

\begin{lem}\label{stretch_loc_con}
 Let $C$ be a simplicial complex such that all link graphs are 2-connected or free-graphs. Then we 
can apply stretchings at local 
2-separators of 2-connected link graphs such that the resulting simplicial complex is locally 
almost 3-connected. 
\end{lem}

Before we prove this, we need a definition. A 2-separator $(x,y)$ in a graph $G$ is \emph{proper} 
unless $G-x-y$ has precisely 
two components 
and one of them is a path and $xy$ is not an edge.

\begin{proof}
 If $C$ has a 2-connected link graph that is not a parallel graph or a 
subdivision of a 3-connected graph, it contains a proper 2-separator and 
we stretch at that 2-separator.
Link graphs at other vertices remain 2-connected or free graphs, respectively.
By \autoref{2-stretch-change} the degree-parameter goes down and hence this process has to stop 
after finitely many steps -- with the desired simplicial complex. 
\end{proof}

\vspace{.3cm}

Until the rest of this section we fix a simplicial complex $C$ with a vertex $v$ such that the link 
$L(v)$ is 2-connected and let $(a,b)$ be a 2-separator of $L(v)$. We 
denote the simplicial complex obtained 
from $C$ by stretching $(a,b)$ at $v$ by $C_2$. 

\begin{rem}\label{inverse2}
 $C$ can be obtained from $C_2$ as follows. First we contract the edges incident with the 
vertex $v$ except for $\bar a$ and $\bar b$. We relabel $\bar a$ by $a$ and $\bar b$ by $b$. 
We obtain some faces of size two, we refer to these faces as \emph{tiny} faces. Then we contract 
all these tiny faces. This gives $C$.  
\end{rem}

We say that an operation, such as contracting an edge, is an \emph{equivalence} for a property, 
such as the existence of a planar rotation systems, if a simplicial complex has that property if 
and only if the simplicial complex after applying this operation has this property.

In {\cite[\autoref*{contr_pres_planar}]{3space1}} it is 
shown that contracting a non-loop edge where the link graph at both endvertices are 2-connected is 
an equivalence for the 
existence of planar rotation systems. Contracting a face of size two is not always an equivalence 
for the 
existence of planar rotation systems but here the contracted faces have the following additional 
property.

A face $f$ incident with only two edges $e_1$ and $e_2$ is \emph{redundant} if there is a vertex 
$v$ incident with $f$ such that in $C/f$ in any planar rotation system of the link graph $L(v)$
at the rotator at $f$, the edges incident with $e_1$ in the link at $v$ for $C$ form an 
interval. (This implies that also the edges incident with $e_2$ in the link at $v$ for $C$ form an 
interval.)

The following is obvious.
\begin{obs}\label{is_eq}
Let $C'$ be obtained from $C$ by contracting a redundant face.
If $C'$ has a planar rotation system, then $C$ has a planar rotation system.
\qed
\end{obs}

\begin{obs}\label{tiny}
Tiny faces (as defined in \autoref{inverse2}) are redundant. 
\end{obs}
\begin{proof}
 Let $X$ be a component of $L(v)-a-b$. Since $L(v)$ is 2-connected, the edges between $a$ and $X$ 
form an interval in any rotator at $a$ for any embedding of $L(v)$ in the plane. The same is true 
for `$b$' in place of `$a$'.
\end{proof}

\begin{lem}\label{PRS2}
The simplicial complex $C$ has a planar rotation system if and only if the simplicial complex 
$C_2$
has a planar rotation system.
\end{lem}

\begin{proof}
 It is shown in {\cite[\autoref*{contr_pres_planar}]{3space1}} that contracting a non-loop edge 
where the link graph at both endvertices are 
2-connected is an equivalence for the 
existence of planar rotation systems.
 
By {\cite[\autoref*{rot_closed_down}]{3space1}} contracting a 
face of 
size two preserves the 
existence of planar rotation 
systems. So contracting tiny faces is an equivalence for the existence of planar rotation systems 
by \autoref{is_eq} and \autoref{tiny}.

Hence all the operation that transform the simplicial complex $C_2$ to $C$ as described in  
\autoref{inverse2} are equivalences. Thus stretching at local 2-separators is an equivalence for 
planar rotation systems. 
\end{proof}

The following is geometrically clear, see \autoref{fig:s_pair}, and we will not use it in 
our proofs. 

\begin{lem}\label{stretch_22_embed} If $C$ embeds in 3-space, then also $C_2$ 
embeds in 3-space.
\qed
\end{lem}

\begin{rem}
 Also the converse of \autoref{stretch_22_embed} is true. 
\end{rem}

\section{Stretching a local branch}\label{s1}

In this section we define stretching local branches and prove basic properties of this operation. 
This operation is necessary for \autoref{main_streching}.  

Given a connected graph $G$ with a cut-vertex $v$, a \emph{branch} at $v$ is a connected component 
$X$ of $G-v$ together with the vertex $v$ (and all edges between $X$ and $v$). A \emph{branch of 
$G$} is a branch at some cut-vertex of $G$. For any branch $B$, there is a unique vertex $v$ such 
that $B$ is a branch at $v$; we refer to that vertex $v$ as \emph{the cut-vertex} of the branch $B$.

Given a 2-complex $C$ with a vertex $v$ such that the link graph $L(v)$ at $v$ is connected and a 
branch $B$ of $L(v)$, the complex $C[B]$ obtained from $C$ by \emph{pre-stretching} $B$ is 
defined as follows, see \autoref{fig:stretch_branch}. 
We denote the cut-vertex of the branch $B$ by $e$; and remark that $e$ is an edge of the simplicial 
complex $C$.

   \begin{figure} [htpb]   
\begin{center}
   	  \includegraphics[height=5cm]{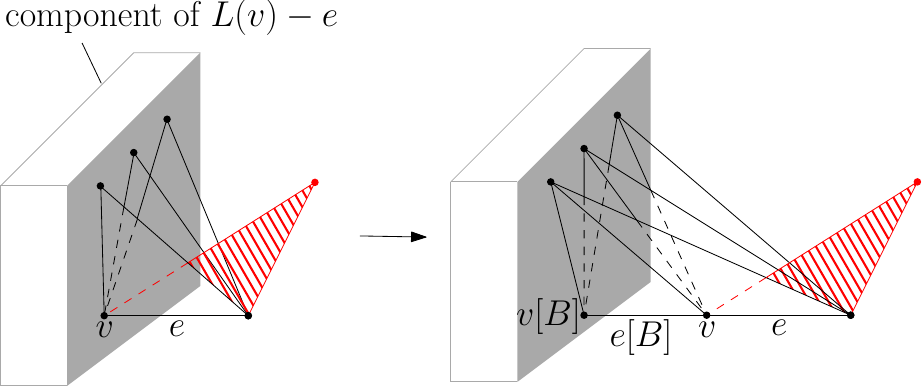}
   	  \caption{The 2-complex on the right is obtained from the 2-complex on the left by 
stretching the branch $B$, which consists of the grey box together with the three faces attaching 
at the grey box. Stretching is defined like pre-stretching but there 
we additionally subdivide faces to make them all have size three.}\label{fig:stretch_branch} 
\end{center}
   \end{figure}

\begin{itemize}
 \item The vertex set of $C[B]$ is that of $C$ together with one new vertex, which we denote 
by $v[B]$, in formulas: $V(C_1)=V(C) \cup \{v[B]\}$;
\item the edge set of $C[B]$ is (in bijection with) the edge set of $C$ together with one 
additional edge, which we denote by $e[B]$, in formulas: $E(C[B])=E(C)\cup \{e[B]\}$;
The incidences between edges and vertices are as in $C$ except for those edges $z\neq e$ of $C$ 
that are vertices of the branch $B$. Such edges are incident with the new vertex $v[B]$ in place of 
$v$, the other endvertex is not changed. The edge $e[B]$ has the endvertices $v$ and $v[B]$;

\item the faces of $C[B]$ are (in bijection with) the faces of $C$; in formulas: 
$F(C[B])=F(C)$. 
The incidences between faces and edges are as in $C$ except for those faces of $C$ that are 
incident with the edge $e$ and are in the link graph $L(v)$ edges of the branch $B$. These faces 
now have size four. They are now additionally incident with the edge $e[B]$. 
\end{itemize}
This completes the definition of pre-stretching the branch $B$ at $v$. 
\emph{Stretching} the branch $B$ is defined the same way except that we additionally subdivide each 
face $f$ of size four once. Namely we add a subdivision-edge between the vertex $v$ and the 
unique vertex of the face that is not in the edge $e$ and different from $v[B]$. Hence for any 
simplicial complex $C$ any stretching at a branch is again a simplicial complex. 

See \autoref{fig:stretch_branch_link} for an 
example illustrating how the link changes at the vertex $v$ and how the link looks like at the 
vertex $v[B]$. 
   \begin{figure} [htpb]   
\begin{center}
   	  \includegraphics[height=4cm]{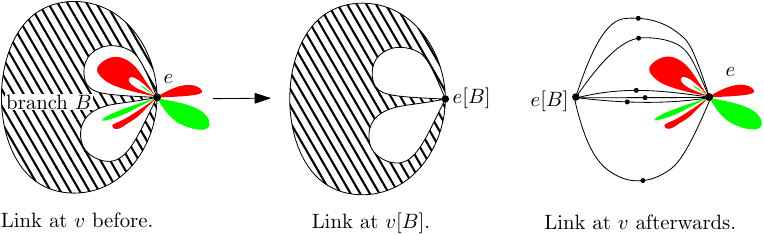}
   	  \caption{On the left we see the link graph of a vertex 
$v$ with a cut-vertex $e$ and a branch $B$. On the right we see the 
link graphs of the two vertices obtained from the link graph of $v$ by stretching 
$B$.}\label{fig:stretch_branch_link} 
\end{center}
   \end{figure}

\vspace{.3cm}
   
Until the rest of this section we fix a simplicial complex $C$ with a vertex $v$ such that the link 
graph $L(v)$ is connected and let $B$ be a branch of that link. We denote the simplicial complex 
obtained 
from $C$ by stretching $B$ by $C[B]$.

\begin{rem}\label{inverse}
The simplicial complex $C$ can be obtained from $C[B]$ as follows. First we contract the edge 
$e[B]$. This makes the 
faces incident with $e[B]$ in $C[B]$ have size two. Then we contract these faces. This gives $C$. 
The 
contracted faces are incident with an edge that is incident with only one other face. 
\end{rem}

\begin{comment}
 We say that an operation, such as contracting an edge, is an \emph{equivalence} for a property, 
such as the existence of a planar rotation systems, if a simplicial complex has that property if 
and only if the simplicial complex after applying this operation has this property.
\end{comment}

   \begin{lem}\label{PRS1}
The simplicial complex $C[B]$ has a planar rotation system if and only if $C$ has a planar rotation 
system.
\end{lem}
   
\begin{proof}
Let $\Sigma$ be a planar rotation system of the simplicial complex $C$. 
We define\footnote{Faces incident with the edge $e$ in $C$ correspond in $C[B]$ either to a single 
face of size three 
or two faces of size three obtained from a face of size four by subdivision. This induces a
bijective map from the faces incident $e$ in $C$ to the faces incident with $e$ in $C[B]$, and an 
injective   
partial map from the faces incident $e$ in $C$ to the faces incident with $e[B]$ in $C[B]$. In 
order to simplify the presentation of the definition we suppress these two maps.} a 
rotation system $\Sigma'$ of the simplicial complex $C[B]$ by taking the same rotator as $\Sigma$ at 
every edge except for $e[B]$ and other new edges, which are incident with two faces.
 At the edges incident with two faces we take the unique cyclic ordering of size two. The rotator 
at the edge $e[B]$ is constructed from the rotator at the edge $e$ for $\Sigma$ by restricting it 
to the faces incident with the edge $e[B]$. 

This rotation system is obviously planar at all vertices of $C[B]$ except for the vertex $v$ and 
$v[B]$. 
We denote by $\Pi$ the rotation system  induced by 
$\Sigma$ of the link graph of $C$ at the vertex $v$. By the construction given 
directly after {\cite[\autoref*{minimal_minor}]{3space1}}, $\Pi$ induces a planar rotation 
system $\Pi_1$ at the branch $B$, and $\Pi$ induces a planar rotation system $\Pi_2$ 
at the minor of the link graph at $v$ in $C$ obtained by contracting $B-e$ to a single 
vertex. 
It is immediate that the rotation system induced by $\Sigma'$ at $v[B]$ is $\Pi_1$, and the 
rotation system induced by $\Sigma'$ at $v$ is $\Pi_2$. 
Hence the rotation system $\Sigma'$ is planar for the simplicial complex $C[B]$.

By {\cite[\autoref*{contr_pres_planar}]{3space1}} contracting an edge preserves the existence of 
planar rotation systems, and by {\cite[\autoref*{rot_closed_down}]{3space1}} contracting a face of 
size two preserves the existence of planar rotation 
systems. Hence by \autoref{inverse} if $C[B]$ has a planar rotation system, then $C$ has a planar 
rotation system.  
\end{proof}

The following is geometrically clear, see \autoref{fig:stretch_branch}, and we will not use it in 
our proofs. 

\begin{lem}\label{stretch_branch_embed} If $C$ embeds in 3-space, then also $C[B]$ 
embeds in 3-space.
\qed
\end{lem}

\begin{rem}
 Also the converse of \autoref{stretch_branch_embed} is true. 
\end{rem}

\section{Increasing local connectivity}\label{s3}

In the first three subsections of this section we define stretchings and prove basic properties; 
these are necessary for \autoref{main_streching}. 
The forth subsection is a preparation for the last subsection, in which we prove 
\autoref{main_streching}, and \autoref{Kura_gen}.  
\subsection{The operation of stretching edges}
Let $C$ be a 2-complex and let $e$ be an edge of $C$ incident with two faces $f_1$ and $f_2$. 
Assume that there is an endvertex $v$ of the edge $e$ such that in any planar rotation system of 
the link graph $L(v)$ at $v$ the edges $f_1$ and $f_2$ are adjacent in the rotator at $e$. 
The complex $C'$ obtained from $C$ by \emph{pre-stretching the edge $e$ in 
the direction of $f_1$ and $f_2$} is obtained from $C$ as follows, see \autoref{fig:stretch_edge}. 
We replace the edge $e$ by two edges new edges $e_1$ and $e_2$, both with the same endvertices as 
$e$. We add a face of size two only incident with $e_1$ and $e_2$. The faces $f_1$ and $f_2$ are 
incident with $e_1$ instead of $e$, all other faces incident with $e$ in $C$ are incident with 
$e_2$ 
instead.
This completes the definition of pre-stretching an edge. \emph{Stretching} an edge is defined the 
same way except that additionally we subdivide the new face of size two 
to obtain a simplicial complex, see \autoref{fig:subdiv}. 

   \begin{figure} [htpb]   
\begin{center}
   	  \includegraphics[height=4cm]{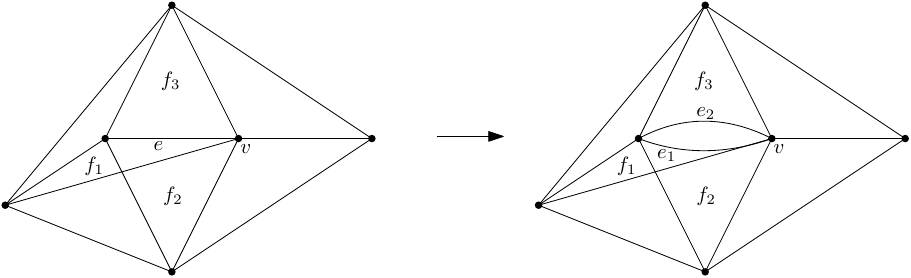}
   	  \caption{The graph on the left defines a simplicial complex by 
adding faces on all cycles of size three. We obtain the 2-complex on the 
right by pre-stretching the edge $e$ in the direction of the faces $f_1$ and 
$f_2$. Its faces are all triangles of the graph on the left except that 
the edge $e_1$ is only incident with the two faces $f_1$ and $f_2$ and the new face  $\{e_1,e_2\}$ 
and the edge $e_2$ is only incident with the face $f_3$ and the new face  
$\{e_1,e_2\}$.}\label{fig:stretch_edge} 
\end{center}
   \end{figure}

      \begin{figure} [htpb]   
\begin{center}
   	  \includegraphics[height=1.3cm]{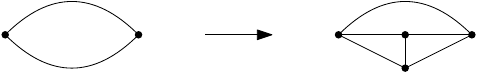}
   	  \caption{Subdivision of a face of size two to a simplicial complex.}\label{fig:subdiv} 
\end{center}
   \end{figure}
      
\begin{eg}
The assumption for stretching an edge $e$ is particularly easy to verify if the link graph $L(v)$ 
is 3-connected. Indeed, then by a theorem of Whitney, we just need to check whether for a 
particular embedding of the link graph $L(v)$ the edges $f_1$ and $f_2$ are adjacent.  
\end{eg}

\begin{rem}\label{rem27}
The inverse operation of pre-stretching an edge $e$ to a face $\{e_1,e_2\}$ is contracting the face 
$\{e_1,e_2\}$ to the edge $e$ as defined in \cite{3space1}.
\end{rem}

\begin{lem}\label{stretch_edge_rotn}
Let $C'$ be obtained from $C$ by pre-stretching an edge $e$. Then $C'$ has a planar rotation system 
if 
and only if $C$ has a planar rotation system.
\end{lem}
   
\begin{proof}
Let $\Sigma$ be a planar rotation system of the 2-complex $C$. 
We denote the two new edges of the 2-complex $C'$ by $e_1$ and $e_2$. 
We obtain a rotation system $\Sigma'$ of the 2-complex 
$C'$ from $\Sigma$ by taking the same rotators at all edges of the 2-complex $C'$ except for $e_1$ 
and $e_2$. 
By assumption, the faces $f_1$ and 
$f_2$ along which we pre-stretch the edge $e$ are adjacent in the rotator at the edge $e$. 
We define the new rotator at the edge $e_1$ to be the rotator of the edge $e$ restricted to the 
adjacent faces $f_1$ and $f_2$ and we add the new face $\{e_1,e_2\}$ in place of the interval 
formed by the deleted faces. Similarly, we define a 
rotator at the edge $e_2$: we delete from the rotator at $e$ the faces $f_1$ and $f_2$ and add the 
face $\{e_1,e_2\}$ in the interval formed by the two deleted faces.
It remains to check that the rotation system $\Sigma'$ is planar. This is immediate at all vertices 
except for the two endvertices of the edge $e$. 
For the two endvertices, note that pre-stretching the edge $e$ has the effect on the link graph as 
coadding an edge at the vertex $e$. As the edges $f_1$ and $f_2$ of the link graphs are 
adjacent, the coaddition can be done within the embeddings  of the link graphs given  by $\Sigma$. 

By {\cite[\autoref*{rot_closed_down}]{3space1}}, contracting a 
face of size two preserves the existence of planar rotation systems. Hence by \autoref{rem27} if 
$C'$ has a planar rotation system, then $C$ has a planar rotation system.  
\end{proof}

The following is geometrically clear, see \autoref{fig:stretch_edge}, and we will not use it in 
our proofs. 

\begin{lem}\label{stretch_edge_embed}
Let $C'$ be obtained from $C$ by stretching an edge $e$. If $C$ embeds in 3-space, then also $C'$ 
embeds in 3-space.
\qed
\end{lem}

\begin{rem}
 Also the converse of \autoref{stretch_edge_embed} is true. 
\end{rem}

\subsection{The operation of contracting edges}   
An edge $e$ in a 2-complex $C$ is \emph{reversible} if the 2-complex $C$
has a planar rotation system if and only if the 2-complex $C/e$ has a
planar rotation system.

A \emph{para-star} is a graph obtained from a family of disjoint parallel graphs by gluing them 
together at a single vertex. 

\begin{lem}\label{is_rev}
Let $e$ be a non-loop edge with endvertices $v$ and $w$ of a 2-complex $C$
such that the link graphs $L(v)$ and $L(w)$ are para-stars and the
vertex $e$ is a maximum degree vertex in both of them.
Then the edge $e$ is reversible.
\end{lem}

\begin{proof}
By {\cite[\autoref*{contr_pres_planar}]{3space1}}, it suffices to show how any 
planar rotation
system $\Sigma$ on the 2-complex $C/e$ induces a planar rotation
system $\Sigma'$ on the 2-complex $C$.
Letting $\Sigma'$ to be equal to $\Sigma$ at all edges of $C$ not
incident with $v$ or $w$, it suffices to show the following.

\begin{sublem}\label{sublem28}
Let $L(v)$ and $L(w)$ be para-stars and let the vertex
$e$ have maximal degree in both of them. Let $L(e)$ be the vertex sum
of $L(v)$ and $L(w)$ along $e$.
For any planar rotation system $\Pi$ of the graph $L(e)$, there are
planar rotation systems of the graphs $L(v)$ and $L(w)$ that are
reverse of one another at the vertex $e$, and otherwise agree
with the rotation system $\Pi$.
\end{sublem}

\begin{proof}
Throughout we assume in the graphs $L(v)$ and $L(w)$, the vertex $e$
is adjacent to any other vertex. This easily implies the general case
by suppressing suitable degree two vertices as rotators
at such vertices are unique.

We prove this by induction on the number of branches of the graph $L(v)$.
The base case is that the graph $L(v)$ is a parallel graph. Then the
graph $L(e)$ is isomorphic to the graph $L(w)$. So a planar rotation
system on the graph $L(e)$ induces a planar rotation
system on the graph $L(w)$. And there is a unique planar rotation
system on the graph $L(v)$ whose rotator at $e$ is reverse to the
rotator at $e$ in that planar rotation system of $L(w)$.

So we may assume that the graph $L(v)$ has at least two branches.
We split into two cases.

{\bf Case 1:} the graph $L(e)$ is disconnected.
We consider $L(e)$ as a bipartite graph with the vertex set of
$L(v)-e$ on the left and the vertex set of $L(w)-e$ on the right.
As every vertex of $L(e)$ is incident with an edge, there are two
vertices of $L(v)-e$ in different connected components of the
bipartite graph $L(e)$. Denote these two vertices by $y$ and $z$.
We obtain $L(v)'$ from $L(v)$ by identifying the vertices $y$ and $z$
into a single vertex. Denote that new vertex by $u$. We denote the
vertex sum of $L(v)'$ and $L(w)$ along $e$ by $L(e)'$. The
graph $L(e)'$ is equal to the graph obtained from $L(e)$ by
identifying the vertices $y$ and $z$. Thus any planar rotation system
of the graph $L(e)$ induces a planar rotation system of the graph
$L(e)'$ by sticking the rotation systems at the vertices $y$ and $z$
together so that the rotator at the new vertex $u$ contains the edges
incident with the vertex $y$ or $z$, respectively, as a
subinterval.
By induction such a rotation system induces planar rotation systems at
the graphs $L(v)'$ and $L(w)$. This planar rotation system at the
graph $L(v)'$ induces a rotation system on the graph $L(v)$
by splitting the rotator at $u$ into the two subintervals for the
vertices $y$ and $z$. This induced rotation system is planar for
$L(v)$ as the rotators for $y$ and $z$ are subintervals of the
rotator for $u$.

{\bf Case 2:} not Case 1, so the graph $L(e)$ is connected. As above,
we consider $L(e)$ as a bipartite graph, and let a planar rotation
system of the graph $L(e)$ be given. Since the left side
has at least two vertices, there is a vertex on the right of the
connected bipartite graph $L(e)$ that has two neighbours on the left.
Pick such a vertex $x$. We pick neighbours $y$ and $z$ of $x$
in $L(v)-e$ such that there are edges $e_y$ between $y$ and $x$ and
$e_z$ between $z$ and $x$ such that these two edges are incident in
the rotator at the vertex $x$.
Let $L(v)'$ be the graph obtained from $L(v)$ by identifying the
vertices $y$ and $z$ to a single vertex. Call that new vertex $u$. We
denote the vertex sum of $L(v)'$ and $L(w)$ along $e$ by $L(e)'$. The graph $L(e)'$ is equal to the 
graph obtained from $L(e)$ by
identifying the vertices $y$ and $z$. The chosen planar rotation
system of the graph $L(e)$ induces a rotation system for the 
graph $L(e)'$ by sticking the rotation systems at the vertices $y$ and
$z$ together so that the rotator at the new vertex $u$ contains the
edges incident with the vertex $y$ or $z$, respectively,
as a subinterval.  By the choice of $y$ and $z$ this rotation system
is planar. By induction this planar rotation system on $L(e)'$ induces
planar rotation systems on the graphs $L(v)'$ and $L(w)$. This planar rotation system at the graph 
$L(v)'$ induces a
rotation system on the graph $L(v)$ by splitting the rotator at $u$
into the two subintervals for the vertices $y$ and $z$. This induced rotation system is planar for 
$L(v)$ as the rotators for $y$
and $z$ are subintervals of the rotator for $u$.
\end{proof}
To summarise the proof of \autoref{is_rev}, we define the planar rotation system $\Sigma'$ for 
the 2-complex $C$ as indicated above, and we choose the rotators at the edges incident with the 
vertices $v$ or $w$ as induced in the sense of \autoref{sublem28} by the rotation system of the link 
graph $L(e)$ at the vertex $e$ of the 2-complex $C/e$.
\end{proof}

\subsection{The definition of stretching}

We say that a simplicial complex $\tilde C$ is obtained from a simplicial complex $C$ by
\emph{stretching}, if it is obtained from $C$ by applying successively operations of the following 
types:
\begin{enumerate}
 \item stretching local branches at connected link graphs;
 \item 2-stretching at local 2-separators of 2-connected link graphs;
 \item stretching edges;
\item contracting reversible edges that are not loops;
\item splitting vertices. 
\end{enumerate}
We also call $\tilde C$ a \emph{stretching} of $C$. 

\begin{lem}\label{planar_rot_preserve}
 Assume $C'$ is a stretching of $C$. Then $C$ has a planar rotation system if and only if $C'$ has 
a planar rotation system.  
\end{lem}

\begin{proof}
In the language introduced above we are to show that all five stretching operations are 
equivalences for the property `existence 
of planar rotation systems'. 
For the first operation it is proved in \autoref{PRS1}, for the second it is proved in 
\autoref{PRS2}, and for the third it is proved in \autoref{stretch_edge_rotn} for pre-stretchings 
of edges, and so the result for stretchings follows. For the forth 
operation it is true by the definition of reversible. 
Splitting vertices is clearly an equivalence for the existence of planar rotation systems.
\end{proof}

\subsection{Increasing the local connectivity a bit}

A 2-complex is \emph{locally almost 2-connected} if all its link graphs are 2-connected or free 
graphs. The following is a key step towards 
\autoref{main_streching}. 

\begin{thm}\label{reduce_to loc_2-con}
Any simplicial complex $C$ has a stretching $C'$ that is a simplicial complex so that $C'$ is 
locally almost 2-connected or $C'$ has a non-planar link graph. 
\end{thm}

Before we prove \autoref{reduce_to loc_2-con}, we need some preparation. 
A \emph{star of parallel graphs} is a graph that is not 2-connected and is obtained from a set of 
disjoint parallel graphs by 
gluing them together at a single vertex. 
\begin{eg}
The only parallel graphs that are stars of parallel graphs are paths. Stars of parallel graphs are 
2-connected para-stars. 
\end{eg}

\begin{lem}\label{stretching_loc_1-cuts}
 Let $C$ be a simplicial complex that is locally connected. Then there is a simplicial complex 
$\tilde{C}$ that is obtained 
from $C$ by stretching local branches such that every link graph of $\tilde C$ is 2-connected 
or a 
star of parallel graphs. 
\end{lem}

\begin{proof}
We will prove this by induction. The base case is that every link graph is 2-connected or a star of 
parallel graphs. Next we consider the case that each link graph has at most one cut-vertex. Let $v$ 
be a vertex of the simplicial complex $C$ such that its link graph has a cut-vertex $e$. Then all 
branches of $e$ are 2-connected graphs. We stretch all branches of $e$ that are not parallel 
graphs, one after the other. Then the link at $v$ becomes a star of parallel graphs and all other 
new link graphs are 2-connected. The old link graphs, those at vertices of $C$ aside from $v$,  
do not change except possibly for subdividing edges\footnote{The vertices $v'$ of $C$ where those 
subdivisions occur are those such that there is an edge $e'$ between $v$ and $v'$ such that 
$e'$ is a vertex of one of the branches we stretch.}. We apply this recursively to all 
link graphs 
with cut-vertices, and so reduce this 
case to the base case. 

Next suppose that there is a vertex $v$ such that its link graph has at least two cut-vertices. 
Let $e_1$ be an arbitrary cut-vertex of that link graph. And let $B$ be a branch of $e_1$ 
containing another cut-vertex $e_2$. Then we stretch $B$. All link graphs at vertices of $C$ 
aside from $v$ are not changed (except for possibly subdividing edges). The vertex $v$ is replaced 
by two new vertices. Each cut-vertex of 
the link graph of $v$ is in precisely one of the two new link graphs, and $e_1$ and $e_2$ are in 
different link graphs. Hence both new link graphs have strictly less cut-vertices than the link 
graph of $v$. Hence we can apply induction (on the sequence of numbers of cut-vertices of link 
graphs, ordered by size and compared in lexicographical order).

\
\end{proof}

\begin{proof}[Proof of \autoref{reduce_to loc_2-con}.]
The \emph{cutvertex-degree} of a simplicial complex $C$ is 
the maximal degree of a cutvertex of a link graph of the simplicial complex $C$.
We prove \autoref{reduce_to loc_2-con} by induction on the cutvertex-degree.
So let $C$ be a simplicial complex with cutvertex-degree $a$. 

We obtain $C_1$ from $C$ by splitting all vertices whose link graphs are disconnected. The 
simplicial complex $C_1$ is locally connected.
By \autoref{stretching_loc_1-cuts} there is a stretching $C_2$ of the simplicial complex $C_1$ such 
that all its link graphs are 2-connected or stars of parallel graphs. 
 
If the cutvertex-degree $a$ is at most three, then all link graphs of $C_2$ are 
2-connected or stars of parallel graphs where the unique cut-vertex has degree at most three. 
Graphs of the second type are always free, see 
\autoref{fig:starstarround}. This completes the proof if  the cutvertex-degree is at most 
three, so from now on let the cutvertex-degree $a$ be at least four. 

We say that a simplicial complex $C$ is \emph{$a$-nice} if all its link graphs are 2-connected, or 
stars of parallel graphs whose cutvertex has degree precisely $a$ or else have maximum degree 
strictly less than $a$. For example the simplicial complex $C_2$ is $a$-nice. 
We say that a simplicial complex $C$ is \emph{$a$-structured} if all its link graphs are 
parallel graphs, or stars of 
parallel graphs whose cutvertex has degree precisely $a$ or else have maximum degree strictly less 
than $a$. For example, every $a$-structured simplicial complex is $a$-nice.

\begin{sublem}\label{be_nice}
Assume $C_2$ is $a$-nice. 
 There is a stretching $C_3$ of $C_2$ that is $a$-structured or else has a non-planar link.
\end{sublem}

\begin{proof}
We prove this sublemma by induction on the degree-parameter as defined in \autoref{s2}.
So let $C_2$ be an $a$-nice simplicial complex such that all $a$-nice simplicial complexes with 
strictly smaller degree-parameter have a stretching that is $a$-structured or has a non-planar 
link.

We may assume that the simplicial complex $C_2$ is not $a$-structured; that is, it has a vertex $v$ 
such that the link graph $L(v)$ is 2-connected but no parallel graph and $L(v)$ has vertex $e$ of 
degree at 
least $a$. Also we may assume that $L(v)$ is planar. 

{\bf Case 1:} the vertex $e$ is not contained in a proper 2-separator of $L(v)$.
Since embeddings of 2-connected graphs in the plane are unique up to flipping at 2-separators by a 
theorem of Whitney, any embedding of the graph $L(v)$ in the plane has the same rotator at the 
vertex $e$ (up to reversing). Take two edges $f_1$ and $f_2$ incident with the vertex $e$ that 
are adjacent in the rotator. 
Now we stretch the edge $e$ of $C_2$ in the direction of the faces corresponding to $f_1$ and 
$f_2$. The link graphs at all vertices of $C_2$ except for $v$ and the 
other 
endvertex $w$ of the edge $e$ of $C_2$ do not change. In the link 
graphs for $v$ and $w$ the vertex $e$ is replaced by two new vertices (and a path of 
length 
two 
joining them), each of strictly smaller 
degree than $e$, as its degree $a$ is at least four. This new simplicial complex $C_3$ has strictly 
smaller degree-parameter than 
$C_2$.

In order to be able to apply induction, we need to show that $C_3$ is $a$-nice. The link graph at 
$v$ is still 2-connected in $C_3$. If the link graph at $w$ in $C_2$ is 2-connected, this is 
still true in $C_3$. Hence it remains to consider the case that it is a star of parallel 
graphs. In this case the vertex $e$ must be the cutvertex of $L(w)$ by the choice of $a$. Then in 
the simplicial complex $C_3$, the link graph at $w$ has maximum degree less than $a$. Thus $C_3$ is 
$a$-nice and we can apply the induction hypothesis. So $C_3$ has a stretching of the desired type, 
and $C_3$ is a stretching of $C_2$. This completes the induction step in this case. 

{\bf Case 2:} not Case 1. Then the vertex $e$ is contained in a proper 2-separator of $L(v)$. Let 
$x$ be the other vertex in that 2-separator. We obtain $C_3$ from $C_2$ by stretching at the 
2-separator $\{e,x\}$. The simplicial complex $C_3$ has strictly smaller degree-sequence 
than 
$C_2$ by \autoref{2-stretch-change}. We verify that $C_3$ is $a$-nice. All link graphs at new 
vertices are still 2-connected in $C_3$. Let $w$ and $w'$ be the endvertices of the edges $e$ and 
$x$ aside from $v$, respectively. Hence it remains so show 
that the link graphs at $w$ and $w'$ in $C_3$ are 2-connected, stars of parallel graphs whose 
cutvertex has degree $a$ or have maximal degree less than $a$. If the link graph at $w$ in $C_2$ is 
2-connected, it is also 2-connected in $C_3$ (as coadding a star preserves 2-connectedness). Hence 
we may assume that the link graph at $w$ in $C_2$ is a star of parallel graphs. and $e$ is its 
cutvertex by the choice of $a$. Then either $L(w)$ is still a star of parallel graphs in $C_3$ or 
else it has maximum degree less than $a$. The same analysis applies to the vertex `$w'$' in place 
of `$w$'. Thus $C_3$ is $a$-nice. So by induction there is a stretching of $C_3$ of the desired 
type, and $C_3$ is a stretching of $C_2$. This completes the induction step, and hence the proof of 
this sublemma. 
\end{proof}

Let $C_3$ be stretching of the simplicial complex $C_2$ as in \autoref{be_nice}. If $C_3$ has a 
non-planar link, we are done. Hence we may assume that the simplicial complex $C_3$ is 
$a$-structured. If $C_3$ has cutvertex-degree less than $a$, we can apply the induction 
hypothesis. 
Hence we may assume that $C_3$ has a vertex $v$ such that the link graph at $v$ is a star of 
parallel graphs whose cutvertex $e$ has degree precisely $a$. Now we show how the property 
`$a$-structured' implies the existence of certain paths, which can then be 
contracted to reduce the cutvertex-degree.

\begin{sublem}\label{path_exists}
 There is a path $P_n$ from the vertex $v$ starting with $e$ to another vertex $w_n$ whose link 
graph is a star of parallel graphs. All link graphs at internal vertices of the path are parallel 
graphs and all edges of the path have the same face-degree. 
\end{sublem}

\begin{proof} 
We build the path $P_n=w_0e_1w_1...e_nw_{n}$ recursively as follows.
We start with $e_1=e$ and $w_0=v$ and let $w_1$ be the endvertex of the edge $e$ aside from $v$. 
Assume we already constructed $w_0e_1w_1...e_iw_{i}$.
If the link graph $L(w_i)$ is a star of parallel graphs we stop and let $i=n$ and 
$w_i=w_n$.
Otherwise by assumption, the link graph $L(w_i)$ must be 2-connected. As the edge $e_i$ has 
degree precisely $a$ the link graph $L(w_i)$ of the $a$-structured simplicial complex $C_3$ is a 
parallel graph.
So the link graph $L(w_i)$ contains a unique vertex except from $e_i$ that has degree larger than 
two, and this vertex has the same degree as the vertex $e_i$. We pick this vertex for $e_{i+1}$. 
Note that 
$e_{i+1}$ is an edge of the simplicial complex $C$. We let $w_{i+1}$ be the endvertex of 
$e_{i+1}$ different from $w_i$.
Note that all edges $e_i$ have the same face-degree by construction.
Since any path\footnote{A \emph{path} in a graph is a sequence alternating between vertices and 
edges such that adjacent members are incident, and all vertices (and edges) are distinct.} in $C$ 
must be finite, it suffices to prove the following:

\begin{fact}
For all $i$ the walk $P_i$ is a path. 
\end{fact}

\begin{proof}
We prove this by induction on $i$. The base case is that $i=1$.
Suppose for a contradiction there is some $j< i$ such that 
$w_i=w_j$. 

{\bf Case 1:} $j=0$: then $e_i$ must be equal to the only vertex of $L(v)$ of the same degree; that 
is, $e_i$ is equal to the edge $e_1$. But then the endvertex $w_{i-1}$ of the edge $e_i$ is equal 
to the vertex $w_1$. This is a contradiction to the induction hypothesis. Hence $w_i$ 
cannot be equal to $w_0$. 

{\bf Case 2:} $j\geq 1$: as in the link graph $L(w_j)$ the only two vertices with the same degree 
as the vertex $e_i$ are $e_{j}$ and $e_{j+1}$, it must be that the edge $e_i$ is equal to one of 
these two edges; that is, the endvertex $w_{i-1}$ of $e_i$ must be equal to 
$w_{j-1}$ or $w_{j+1}$. 
The vertex $w_{j-1}$ cannot be an option by the induction hypothesis. 
Similarly, the vertex $w_{j+1}$ cannot be an option by the induction hypothesis if $j+1< i-1$. So 
$j+2\geq i$, so 
$j=i-2$ or $j=i-1$. 
Since the simplicial complex $C$  has no loops or parallel edges any three consecutive 
vertices on $P_i$, such as $w_{i-2}$, $w_{i-1}$ and $w_i$, are distinct. Hence neither  $j=i-2$ nor 
$j=i-1$ 
are possible. Thus we have also reached a 
contradiction in this case. Hence the vertex $w_i$ is distinct from all previous vertices on the 
walk $P_i$.  
\end{proof}

\end{proof}

Given a path $P_n$ with endvertex $w_n$ as in \autoref{path_exists}, whose link graph $L(w_n)$ at 
$w_n$ is a star of parallel 
graphs, denote the (unique) cut-vertex of the link graph $L(w_n)$ by $x$.
\begin{sublem}\label{cutvx}
 The cut-vertex $x$ is equal to the last edge $e_n$ on the path $P_n$.
\end{sublem}

\begin{proof}
We denote the degree of the cut-vertex $x$ by $a'$. By the definition of $a$, we have, $a'\leq a$. 

On the other hand by \autoref{path_exists} the vertices $e$ and $e_n$ have the same degree in 
the graphs $L(v)$ and $L(w_n)$, and this 
degree is equal to $a$ by the choice of the vertex $v$. 
As $L(w_n)$ is a star of parallel graphs with a 
cut-vertex, the degree of the cut-vertex $x$ is strictly larger than the degree of any other vertex 
of 
$L(w_n)$.
Hence it must be that $a=a'$ and $x=e_n$.
\end{proof} 
 
By \autoref{path_exists} and \autoref{cutvx}, there is a set of vertex-disjoint paths in $C_3$ 
such that any of their endvertices has a link graph that is a star of parallel graphs whose 
cutvertex has degree $a$. All internal vertices of these paths are parallel graphs. And by taking 
this collection maximal, we ensure that any vertex whose link graph is a star of parallel 
graphs whose 
cutvertex has degree $a$ is an endvertex of one of these paths. We denote the set of these paths by 
$\Pcal$. 
We obtain the 2-complex $C_4$ from $C_3$ by contracting all edges on these paths of $\Pcal$.  
Contracting the edges on the paths recursively, we note at each step that these edges are 
reversible by \autoref{is_rev}. 
At all vertices of $C$ except 
for those vertices on the paths, the two $2$-complexes $C_3$ and $C_4$ have the same link graphs. 
In 
addition, $C_4$ has the contraction vertices, one for each of the vertex-disjoint paths. These 
link graphs are the vertex-sum of the link graphs at the vertices on its path, see 
{\cite[\autoref*{sec_vertex_sum}]{3space1}} for background on vertex-sums.
So the link graph at a new contraction vertex is (isomorphic to) the vertex sum of the link graphs 
at the two endvertices plus various subdivision vertices coming from the parallel graphs at 
internal vertices of the path. By \autoref{cutvx}, each 
of 
these vertices in the link graph has degree strictly less than $a$. Hence all new contraction 
vertices have maximum degree less than $a$. Hence the cutvertex-degree of $C_4$ is 
 strictly smaller than $a$. 
 So the 2-complex $C_4$ satisfies all the conditions to apply the induction hypothesis except that 
it may not be a simplicial complex as it may have edges that are loops or parallel edges. 

Now we 
show how we can stretch local branches of $C_3$ to get a simplicial complex $C_3'$ so that the 
simplicial complex $C_4'$ obtained from $C_3'$ by contracting all the paths in $\Pcal$ is a 
simplicial complex. We obtain $C_3'$ from $C_3$ by stretching at each endvertex of a path in $\Pcal$ 
all the branches and at each interior vertex of a path in $\Pcal$ we stretch at the 2-separator 
consisting of the two branching vertices of its parallel graph. We obtain $C_4'$ from $C_3'$ by 
contracting the above defined family of paths $\Pcal$. It is straightforward to check that $C_4'$ 
is a simplicial complex -- and is a stretching of $C$ with smaller cutvertex-degree. This completes 
the induction step, and hence this proof. 
\end{proof}

\subsection{Proofs of \autoref{Kura_simply_con} and \autoref{main_streching}}

We conclude this section by proving the following theorems mentioned in the Introduction.

\begin{proof}[Proof of \autoref{main_streching}]
Let $C$ be a simplicial complex. Recall that  \autoref{main_streching} says there is a simplicial 
complex $C'''$ obtained from $C$ by stretching 
so that $C'''$ is 
locally almost 3-connected and stretched out or $C'''$ has a non-planar link; moreover $C$ has a 
planar rotation system if and 
only if $C'''$ has a planar rotation system.

 By \autoref{reduce_to loc_2-con} there is a stretching $C'$ 
of $C$ that is a simplicial complex that is locally almost 2-connected or has a non-planar link. 
As we are done otherwise, we may assume that $C'$ is locally almost 2-connected.
By 
\autoref{stretch_loc_con} 
there is a stretching 
$C''$ of $C'$ that is a locally almost 3-connected simplicial complex.

 \begin{sublem}\label{make_strected_out}
 Let $C''$ be a locally almost 3-connected simplicial complex. Then there is a 
stretching $C'''$ of $C''$ that has additionally the property that it is stretched out.
\end{sublem}

\begin{proof}
We say that an edge of face-degree two is \emph{stretched out} if it has one endvertex that is not 
a subdivision of a 3-connected graph or a parallel graph whose branch vertices have degree at least 
three. Note that a simplicial complex in which every edge of degree two is stretched out is 
stretched out itself.
We prove this sublemma by induction on the number of edges of degree two that are not stretched 
out. 
So assume there is an edge $e$ that is not stretched out. Let $v$ be one of its endvertices.

{\bf Case 1:} the link graph at $v$ is a parallel graph whose two branch vertices $x_1$ and $x_2$ 
have degree at least three. Then we stretch at the 2-separator $(x_1,x_2)$ at $v$. 
This gives a simplicial complex $\tilde C$ that in addition to the vertex $v$ has also one new 
vertex for 
every component of $L(v)-x_1-x_2$. The link graphs at these new vertices are cycles. Hence every 
edge of degree two incident with these new vertices is stretched out. Thus $\tilde C$ has strictly 
less 
edges of degree two that are not stretched out.

{\bf Case 2:} the link graph at $v$ is a subdivision of a 3-connected graph. 
Then the vertex $e$ of $L(v)$ is contained in a subdivided edge. Let $P$ be the path of that 
subdivided edge, 
and let $x_1$ and $x_2$ be its endvertices. Then we stretch at the 2-separator $(x_1,x_2)$ at $v$. 
The rest of the analysis is analogue to Case 1. This completes the proof of the sublemma. 
\end{proof} 
 
By \autoref{make_strected_out} we may assume that $C$ has a stretching $C'''$ that is a locally 
almost 3-connected 
and stretched out simplicial complex.  
The `Moreover'-part follows from the fact that $C'''$ is a stretching of $C$ as shown in 
\autoref{planar_rot_preserve}. This completes the proof. 
 \end{proof}

 \begin{proof}[Proof of \autoref{Kura_simply_con}]
Let $C$ be a simply connected simplicial complex. 
Recall that \autoref{Kura_simply_con} says that $C$ has an embedding in 3-space if and 
only if $C$ has no stretching that has a space minor in $\Zcal\cup \Tcal$. 
By {\cite[\autoref*{combi_intro}]{3space2}} $C$ 
is embeddable in 3-space if and only if it has a 
planar rotation system.

By \autoref{main_streching} there is a simplicial complex $C'$ that is a stretching of $C$. 
Moreover $C$ has a planar rotation system if and only if $C'$ has a planar rotation system. 
By that theorem either the simplicial complex $C'$ has a non-planar link or it is locally almost 
3-connected and stretched out. In the first case, 
by 
Kuratowski's theorem, {\cite[\autoref*{reduce_to_looped_cone1}]{3space1}} and 
{\cite[\autoref*{gen_cone}]{3space1}}, the simplicial complex $C'$ has a minor in the finite list 
$\Zcal$ -- so the theorem is true in this case. In the second case by \autoref{Kura_almost} $C'$ has 
a planar rotation system if and only if it has no space minor in 
$\Zcal\cup \Tcal$. 
This completes the proof. 
  \end{proof}

\section{Algorithmic consequences}\label{algo_sec}  

Our proofs give a quadratic algorithm that verifies whether a given 2-dimensional 
simplicial complex has a planar rotation system. This gives a quadratic algorithm that checks 
whether a given 2-dimensional simplicial complex has an embedding in a (compact) 
orientable 3-manifold by 
{\cite[\autoref*{is_manifold}]{3space2}} (for the general, not necessarily orientable, case see 
{\cite[\autoref*{non-or}]{3space2}}). In particular, for simply connected 2-complexes this 
gives 
a quadratic algorithm 
that tests embeddability in 3-space by Perelman's theorem. The algorithm has several components. 
Next we explain them and prove the relevant lemmas afterwards. 

\begin{enumerate}
 \item The locally almost 3-connected and stretched out case. The corresponding fact in the paper 
is \autoref{41_2}. This clearly has a linear time algorithm. 
\item Reduction of the locally almost 3-connected case to the locally almost 3-connected and 
stretched out case. The corresponding fact in the paper 
is \autoref{make_strected_out}. This clearly has a linear time algorithm. 
 \item Reduction of the locally almost 2-connected case to the locally almost 3-connected case. 
 The corresponding fact in the paper 
is \autoref{stretch_loc_con}. To analyse the running time, we do this step slightly differently 
than in the paper. First we compute a Tutte-decomposition\footnote{A Tutte-decomposition is a 
decomposition of a graph (or matroid) into its 3-connected components along 2-separators. In the 
special case of graphs it is also known as the \emph{SPQR tree}.} at every 2-connected link graph. 
This 
tells us precisely how we can stretch that vertex along 2-separators. Doing these stretchings at 
different vertices may affect the link graphs at other vertices. Indeed, it may affect other 
vertices in that we coadd stars at their link graphs. 
However, once a link graph is a 
subdivision of a 3-connected graph, it will stay that. So the vertices we may have to look at 
multiple times are vertices where the link graphs are parallel graphs. But if we need to stretch 
there again, the maximum degree goes down. Using \autoref{coadd_star} below, it is straightforward 
to show that this step can be done 
in linear time. 
 \item Reduction of the general case to the locally almost 2-connected case.   The corresponding 
fact in the paper is \autoref{reduce_to loc_2-con}.

This is done by 
recursion on the cutvertex-degree $a$. So let us analyse the step from $a$ to $a-1$ in detail. 
The input is a simplicial complex $C$ and we measure its size by $\sum (deg(e)-2)$, where the sum 
ranges over all edges $e$ of $C$ of degree at least three. We refer to that sum as the \emph{degree 
parameter}.\footnote{ 
We remark that at edges of degree at most two the compatibility conditions for planar rotation 
systems is always satisfied and hence we do not need to take them into account in the definition of 
the degree parameter.} 

The stretching related to \autoref{stretching_loc_1-cuts} can be done in linear time as computing 
the block-cutvertex-tree of link graphs can be done in linear time. For the part 
corresponding to \autoref{be_nice} we compute the stretching via Tutte-decompositions as in step 3 
explained above, and then we check for planarity for each 3-connected link graph. If it is planar, 
we remember a planar rotation system and if we stretch later an edge incident with that vertex at 
the other endvertex we check whether this stretching is compatible with the chosen planar rotation 
system. This can be done in 
linear time. The construction of the set $\Pcal$ of paths can clearly be done 
in linear time.
Hence the whole recursion step from $a$ to $a-1$ just takes linear time. The output is the 
simplicial complex $C_4'$. 

However, with the current argument, the degree parameter of $C_4'$ might be larger 
than the degree parameter of the input $C$. Indeed, stretching a local branch may increase the 
degree parameter. Hence here 
we explain how we modify the construction of the simplicial complex $C_4'$ so that the degree 
parameter does not increase. 
First note that none of the stretching operations except for stretching a branch increases the 
degree parameter, compare \autoref{degpar}. 
We obtain $C''$ from $C_3'$ by contracting all edges 
$e[B]$ of degree at least three that 
were added by stretching a local branch, and contracting the resulting faces of size two (this 
has the effect of reversing the stretching operations at those edges $e[B]$); additionally we 
stretch so that no edge of degree two has both endvertices on the same path -- similarly as in the 
construction of stretched out in \autoref{make_strected_out} (this ensures that edges of degree two 
do not make a problem 
later. This does 
not increase the degree parameter). It is easy to see that $C''$ is a simplicial complex and that 
$C$ has a planar rotation system if and only if $C''$ has one. 
Each path $P\in \Pcal$ of $C_3'$ contracts onto a closed trail of $C''$. We obtain $C_3''$ from 
$C''$ by stretching branches for each closed trail $P\in \Pcal$ as follows.

{\bf Case 1:} in the simplicial complex $C''$ the trail $P$ has at least one internal vertex. 
Denote the edge of $P$ incident with $v$ by 
$e_1$ and the edge of $P$ incident with $w$ by 
$e_2$. Then we stretch at the link graphs of 
 $v$ and $w$ all branches at $e_1$ and $e_2$, respectively.

{\bf Case 2:} not Case 1. Note that $P$ must consist of at least one edge by 
\autoref{path_exists}. And that edge is not a loop as $C''$ is a simplicial complex.
So $P$ consists of a single edge $e$. 
Then in the simplicial complex $C''$ there is no edge in parallel to $e$. We stretch all branches 
of the link graph at  $v$ at the vertex $e$ (but not for $w$).

We obtain $C_4''$ from  $C_3''$ by contracting all $P\in \Pcal$. It is 
straightforward to check that $C_4''$ is a simplicial complex and that the degree parameter of 
$C_4''$ is at 
most that of $C$, compare \autoref{C4}. By construction the cutvertex degree of $C_4''$ is 
strictly smaller than that of $C$. So $C_4''$ is a suitable output of the recursion step. As each 
recursion step takes linear time, all of them together take at most quadratic time. 
\end{enumerate}
This completes the description of the algorithm. 

\subsection{Some lemmas for the algorithm above}

Here we prove the lemmas referred to in the beginning of \autoref{algo_sec}. 

The \emph{degree-parameter} of a graph $G$ is $\sum(deg(v)-2)$, where the sum ranges 
over all vertices. 

\begin{lem}\label{coadd_star}
 Coadding a star at a vertex $v$ preserves the degree parameter.
 \end{lem}

\begin{proof}
 Let $k$ be the degree of the center of the coadded star and $v_1, .., v_k$ be the leaves of the 
coadded star. The degree parameter of the graph before minus the degree parameter after the 
coaddition is:
\[
deg(v)-2 -\left(\sum_{i=1}^k (deg(v_i)-2)+(k-2)\right)
\]
As $deg(v)=\sum_{i=1}^k deg(v_i)-k$ the above sum evaluates to zero, completing the proof.
\end{proof}

\begin{lem}\label{degpar}
All stretching operations except for possibly stretching a local branch do not increase the degree 
parameter.
\end{lem}
 
\begin{proof}
 This lemma is immediate for splitting vertices and contracting reversible non-loops.
 If we stretch an edge, first note that pre-stretching does not change the degree parameter as 
$deg(e)=deg(e_1)+deg(e_2)-2$ if $e$ is stretched to $e_1$ and $e_2$.  Then note that subdivisions 
of faces do not change the degree parameter. 

The fact that 2-stretching does not change the degree parameter is proved similarly as 
\autoref{coadd_star}. 
\end{proof}
 
 \begin{lem}\label{C4}
The complex $C_4''$ is a simplicial complex whose degree parameter is not larger 
than that of $C$.
\end{lem}
 
\begin{proof}
 As mentioned above $C''$ is a simplicial complex. By \autoref{degpar} the degree parameter of 
$C''$ is not larger than that of $C$.
Hence it suffices to show for each trail $P\in \Pcal$ that the construction given in the 
two cases plus the contraction afterwards preserves being a simplicial complex and does not 
increase the degree parameter. 

First we treat Case 1; that is, $P$ has an internal vertex. In the construction of $C''$ we never 
contract an edge incident with an internal vertex of $P$ that is not on the path $P$. Hence $P$ is 
a path in $C''$ or a cycle. Then we stretch all the branches at the local cutvertices $e_1$ and 
$e_2$ in the link graphs at $v$ and $w$, respectively. The sum of degrees of the new edges $e[B]$ 
is at most $deg(e_1)+deg(e_2)$. So the degree parameter increased by at most $2deg(e_1)-4$ (as 
$deg(e_1)=deg(e_2)$). Then we contract all edges on $P$. This decreases the degree parameter at 
least by that amount. Thus in total the degree parameter does not increase. The stretching 
before 
the contraction ensures that we do not create parallel edges or loops. 

The analysis in Case 2 is similar. 
\end{proof}

   \section*{Acknowledgement}

I thank Radoslav Fulek for pointing out an error in an earlier version of this paper.

\bibliographystyle{plain}
\bibliography{literatur}

\end{document}

%% file: minimalJ.tex
\usepackage{epsfig}
\usepackage{graphicx}
\usepackage{amsbsy}
\usepackage{amsmath}
\usepackage{amsfonts}
\usepackage{amssymb}
\usepackage{textcomp}
\usepackage{hyperref}
\usepackage{aliascnt}

\newcommand{\mcm}[3]{\newcommand{#1}[#2]{{\ensuremath{#3}}}} 

\mcm{\tuple}{1}{\langle #1 \rangle}
\mcm{\name}{1}{\ulcorner #1 \urcorner}
\mcm{\Nbb}{0}{\mathbb{N}}
\mcm{\Zbb}{0}{\mathbb{Z}}
\mcm{\Rbb}{0}{\mathbb{R}}
\mcm{\Cbb}{0}{\mathbb{C}}
\mcm{\Qbb}{0}{\mathbb{Q}}
\mcm{\Acal}{0}{\cal A}
\mcm{\Bcal}{0}{\cal B}
\mcm{\Ccal}{0}{\cal C}
\mcm{\Dcal}{0}{\cal D}
\mcm{\Ecal}{0}{\cal E}
\mcm{\Fcal}{0}{\cal F}
\mcm{\Gcal}{0}{\cal G}
\mcm{\Hcal}{0}{\cal H}
\mcm{\Ical}{0}{\cal I}
\mcm{\Jcal}{0}{\cal J}
\mcm{\Kcal}{0}{\cal K}
\mcm{\Lcal}{0}{\cal L}
\mcm{\Mcal}{0}{\cal M}
\mcm{\Ncal}{0}{\cal N}
\mcm{\Ocal}{0}{{\cal O}}
\mcm{\Pcal}{0}{{\cal P}}
\mcm{\Qcal}{0}{{\cal Q}}
\mcm{\Rcal}{0}{{\cal R}}
\mcm{\Scal}{0}{{\cal S}}
\mcm{\Tcal}{0}{{\cal T}}
\mcm{\Ucal}{0}{{\cal U}}
\mcm{\Vcal}{0}{{\cal V}}
\mcm{\Wcal}{0}{{\cal W}}
\mcm{\Xcal}{0}{{\cal X}}
\mcm{\Ycal}{0}{{\cal Y}}
\mcm{\Zcal}{0}{{\cal Z}}
\mcm{\Mfrak}{0}{\mathfrak M}

\mcm{\restric}{0}{\upharpoonright}
\mcm{\upset}{0}{\uparrow}
\mcm{\onto}{0}{\twoheadrightarrow}
\mcm{\smallNbb}{0}{{\small \mathbb{N}}}
\DeclareMathOperator{\preop}{op}
\mcm{\op}{0}{^{\preop}}

{\begin{array}{c}
\setlength{\unitlength}{1em}}%
{\end{array}}

\usepackage{amsthm}

\newcommand{\theoremize}[2]{\newaliascnt{#1}{thm} \newtheorem{#1}[#1]{#2} \aliascntresetthe{#1}}

\theoremstyle{plain}
\newtheorem{thm}{Theorem}[section]
\theoremize{lem}{Lemma}
\theoremize{skolem}{Skolem}
\theoremize{fact}{Fact}
\theoremize{sublem}{Sublemma}
\theoremize{claim}{Claim}
\theoremize{obs}{Observation}
\theoremize{prop}{Proposition}
\theoremize{cor}{Corollary}
\theoremize{que}{Question}
\theoremize{oque}{Open Question}
\theoremize{con}{Conjecture}

\theoremstyle{definition}
\theoremize{dfn}{Definition}
\theoremize{rem}{Remark}
\theoremize{eg}{Example}
\theoremize{exercise}{Exercise}
\theoremstyle{plain}